\documentclass[a4paper,12pt]{article}

\usepackage[utf8]{inputenc}
\usepackage{amsthm,amsmath,stmaryrd,bbm,hyperref,geometry,color,xcolor}
\usepackage{amssymb}
\usepackage[english]{babel}
\usepackage{graphicx}
\usepackage{amsfonts,amssymb}
\usepackage{verbatim}
\usepackage{enumitem}
\usepackage[all]{xy}
\usepackage{caption}
\usepackage{subcaption}

\newcommand{\po}{\left(}
\newcommand{\pf}{\right)}
\newcommand{\co}{\left[}
\newcommand{\cf}{\right]}
\newcommand{\cco}{\llbracket}
\newcommand{\ccf}{\rrbracket}
\newcommand{\E}{\mathbb E}
\newcommand{\R}{\mathbb R} 
\newcommand{\T}{\mathbb T} 
\newcommand{\C}{\mathcal C}

\newcommand{\N}{\mathbb N}

\newcommand{\dd}{\text{d}}
\newcommand{\na}{\nabla}

\newtheorem{thm}{Theorem}

\newtheorem{lem}[thm]{Lemma}

\newtheorem{prop}[thm]{Proposition}
\newtheorem{cor}[thm]{Corollary}
\newtheorem{rem}{Remark}

\title{Switched diffusion processes for non-convex optimization and saddle points search}
\author{Lucas Journel, Pierre Monmarché}
\date{September 2022}

\begin{document}

\maketitle

\begin{abstract}
    We introduce and investigate  stochastic processes designed to find local minimizers and saddle points of non-convex functions, exploring the landscape more efficiently than the standard noisy gradient descent. The processes switch between two behaviours, a noisy gradient descent and a noisy saddle point search. It is proven to be well-defined and to converge to a stationary distribution in the long time. Numerical experiments are provided on low-dimensional toy models and for  Lennard-Jones clusters.
\end{abstract}


\section{Introduction}

\subsection{Overview}

 This work addresses the issue of finding the local minimizers and the saddle points 
of a given  non-convex potential $U:\R^d\to\R_+$. These questions are ubiquituous in a wide range of scientific fields. In particular, in molecular dynamics, $U$ being the energy of a molecular system, the minimizers are the metastable state of the system, while the saddle points are the transition states between the former. As we will see, we will design a stochastic algorithm which targets both local minimizers and saddle points, in such a way that each of this task benefits from the other. In particular, the algorithm is suitable even as an optimisation tool, when the goal is only to find the global minimizers of $U$.

In high dimension, this problem is challenging, as an exhaustive search is impossible and running local search algorithms started from random initial conditions only works when it is possible to sample starting points that have a reasonable probability to hit the basin of attraction of the points of interest (for a gradient descent or a saddle point search), which is often not the case when these basins  are concentrated around an unknown manifold of dimension much smaller than $d$.

A classical solution in order to visit the whole space while taking into account the energy landscape (hence focusing on low-energy areas) is to follow the overdamped Langevin diffusion process, which is the Markov process solving
\begin{equation}\label{eq:Langevin}
    \dd X_t = -\na U(X_t)\dd t + \sqrt{2\varepsilon}\dd B_t\,,
\end{equation}
where $B$ is a $d$-dimensional Brownian motion and $\varepsilon>0$ is a temperature parameter. When $\varepsilon$ is small, this is a noisy perturbation of a deterministic gradient descent (contrary to the stochastic gradient descent algorithm, here the noise is voluntarily added, it has a constant intensity and does not come from a Monte Carlo approximation of $\na U$). Under standard growth conditions on $U$ at infinity, the process is ergodic with respect to the measure
\[
\mu_\varepsilon(\dd x) = \frac{e^{-U(x)/\varepsilon}\dd x}{\int_{\R^d} e^{-U(y)/\varepsilon}\dd y}\,,
\]
i.e. the time spent in any given domain is asymptotically proportional to its probability with respect to $\mu_\varepsilon$. Hence, the temperature $\varepsilon$ has to be taken small enough to counteract the entropic effect due to the high dimension $d$ and ensure that the vicinity of  minimizers have a high probability. Then, one can use the point of the trajectory with the lowest energy as a starting point for a local gradient descent (or a variant, like the stochastic gradient descent). However, as $\varepsilon$ vanishes, it is well known that 
\eqref{eq:Langevin} is metastable, since transitions between different potential wells occur in a time of order roughly $e^{c/\varepsilon}$ for some $c>0$, due to the energy barriers induced by low-probability regions interspersing high-probability ones. This makes the convergence of the law of the process towards its equilibrium, and thus the exploration, very slow, which makes this method fail in practice.

Besides, according to large deviations theory, more precisely to the Freidlin-Wentzell results on the low noise behaviour of diffusion processes, it is known that, as $\varepsilon$ vanishes, the process \eqref{eq:Langevin} is likely to exit from a potential well through the saddle point with the lowest energy level on the boundary of the well. Instead of waiting for an unlikely deviation of the Brownian motion to lead the process \eqref{eq:Langevin} to this saddle point, this naturally motivates the definition of a process which actively looks for it. More precisely, the main object of this work is a process which switches randomly between two dynamics, the noisy gradient descent \eqref{eq:Langevin} and a noisy saddle points search.

The rest of this work is organized as follows. In the remainder of the introduction, we present two saddle point search algorithm in Section~\ref{s-sec:intro-dertemini}, and we define the stochastic processes we are going to investigate. In Section~\ref{sec:Theory}, we conduct a theoretical study of the processes, and in Section~\ref{sec:Num} a numerical study.

\subsection{Deterministic saddle point algorithms}\label{s-sec:intro-dertemini}

Since saddle point search is less standard than optimization, let us first present some existing deterministic algorithms on this topic. A first class of methods, based on reaction paths, requires the knowledge of two local minimizers, and then finds a path of minimal elevation between them, which passes through a saddle point see e.g. \cite{henkelman2002methods,E_2002}. This doesn't correspond to our context, where the local minimizers are unknown. A second class of methods, considered e.g. in  \cite{Non-conservative-GAD,high-index-GAD,minimisation-formulation,ISD-GAD,Weiman2012,Bofill2015,Gu2018}  and references within and  which will be the one of interest for us, relies on local walkers, analogous of the gradient descent for optimization, i.e. solutions of some ODEs for which saddle points of $U$ are stable equilibrium.   Given some $x\in\R^d$, write:
\[
\lambda_1(x) \leqslant \lambda_2(x) \leqslant \cdots \leqslant \lambda_d(x)
\]
for the ordered eigenvalues of $\na^2U(x)$, and suppose that we are given $v_1(x),\dots,v_d(x)\in\R^d$ such that $v_i(x)$ is an eigenvector of $\na^2 U(x)$ associated to $\lambda_i(x)$ (in other words, when the eigenvalues are not all simple, we assume that we have an arbitrary rule to select a given eigenbasis, for instance we take the basis obtained as the limit of the Jacobi algorithm). The idealised saddle dynamic (ISD) is the solution of 
\[\dot x_t = - \po I - 2 v_1(x_t) v_1(x_t)^T\pf \na U(x_t)\,.\]
In other words, $x_t$ follows a gradient descent, except in the direction $v_1(x_t)$ where a reflection is performed, i.e. the process follows a gradient ascent in this direction. Notice that in general $x\mapsto v_1(x)$ is not continuous so the existence of this process is unclear and may be restricted to some parts of the space or to a finite time interval.

A variation of the ISD is the so-called Gentlest Ascent Dynamic (GAD) $(x_t,v_t)\in\R^d \times \mathbb S^{d-1}$ solving
\begin{equation*}
    \left\{ \begin{array}{lll}
        \dot x_t  & = & - \po I - 2 v_t v_t^T\pf \na U(x_t)  \\
        \eta \dot v_t & = & -(I-v_tv_t^T)\na^2U(x_t)v_t\,,
    \end{array}  \right.
\end{equation*}
where $\eta>0$ is a fixed parameter. If $\na^2 U(x_t)$ were fixed in the second equation, this would be a gradient descent on $\mathbb S^{d-1}$ for the Rayleigh quotient $v\mapsto -v^T \na^2 U(x_t) v$. As $\eta$ vanishes, formally, we recover the ISD. Contrary to the later, the GAD is always well-defined for all initial conditions and all times.

The following results on the ISD and GAD are from \cite[Theorems 2 and  3]{ISD-GAD} and \cite{Non-conservative-GAD}.

\begin{prop}\label{prop:deterministicISD-GAD}
\begin{itemize}
    \item Any critical point of $U$ is an equilibrium points for the ISD. It is stable if and only if it is a index-1 saddle point.
    \item Any $(x,v)\in\R^d\times\mathbb S^{d-1}$ with $x$ a critical point of $U$ and $v$   an eigenvalue of $\na^2U$ is an equilibrium point of the Gentlest Ascent dynamic. It is stable if and only if $x$ is a index-1 saddle point and $v$ is associated to the negative eigenvalue.
    \item If $\Omega\subset\R^d$ satisfies that for all $x\in \Omega$, we have $\lambda_1(x)<0<\lambda_2(x)$, then there exists a unique saddle point $z_*$ in $\Omega$, and for all $x_0\in\Omega$, there exists a unique solution to the $ISD$ which converges exponentially fast towards $z_*$.
    \item Under the same assumption, there is $\delta >0$ such that if $(x_0,v_0)\in\Omega\times \mathbb S^{d-1}$, and $|v_0-v_1(x_0)|<\delta$, then there exists a unique solution to the $GAD$ which converges exponentially fast towards $(z_*,v_1(z_*))$.
\end{itemize}
\end{prop}

The convergence results are only local, and one cannot expect more in general since the ISD may have attracting singularities \cite{ISD-GAD}. In the case of the GAD, with a potential of the form $\sum_{i=1}^dU_i(x_i)$, and an initial condition $v_0\in \text{vect}(e_1,...e_i)$, $v_t$ will stay orthogonal to $e_j$, $j>i$, and hence the process will never get close to an equilibrium point with an eigenvector not in $\text{vect}(e_1,...e_i)$. As in the optimization case where the overdamped Langevin diffusion \eqref{eq:Langevin} allows for a global exploration while the gradient descent only converges locally, it is thus natural to add a small Brownian noise to the deterministic ISD and GAD. The resulting processes will explore eventually the whole space but,  thanks to Proposition~\ref{prop:deterministicISD-GAD}, will still be attracted locally by saddle points (notice that the Freidlin-Wentzell theory on random perturbation of ODE does not apply directly in the ISD case because of the discontinuity of $x\mapsto v_1(x)$). 

\subsection{The stochastic processes}\label{s-sec:def-process}

For simplicity, from now on, we will restrict ourselves to the case of a function defined on compact manifold (specifically the periodic torus). Indeed, in practical situation, local minima and saddle points are located in a compact set. Most of the definitions and arguments below would be easily adapted to $\R^d$ provided suitable modifications outside a given compact, in particular to ensure the stability of the processes.

Write $\T^d$ for the d-dimensional torus and fix some function $U\in\mathcal C^2(\T^d,\R_+)$. As in the previous section, we write $\lambda_i(x)$ the ordered eigenvalues of $\na^2 U(x)$, and we consider functions $x\mapsto v_i(x)$ such that $v_i(x)$ is an eigenvector of $\na^2 U(x)$ associated to $\lambda_i(x)$ for all $x\in\T^d$ and $i\in\cco 1,d\ccf$. Moreover, we assume that the functions $v_i$ are measurable for all $i\in\cco 1,d\ccf$, which is indeed possible:

\begin{lem}
For all $i\in\cco 1,d\ccf$ there exists a measurable function $v_i :\T^d\rightarrow \mathbb S^{d-1}$ such that $v_i(x)$ is and eigenvector of $\na^2 U(x)$ associated to $\lambda_i(x)$ for all $x\in\T^d$.
\end{lem}
\begin{proof}
Given a symmetric matrix $H$ and $(e_1,...,e_d)$ the canonical basis of $\R^d$, the function $H\mapsto (u_1,\mu_1,...,u_d,\mu_d)$, where $u_i$ is the limit of the gradient descent for the Rayleigh quotient starting from $e_i$ and $\mu_i$ is the associated eigenvalue, is measurable. At least one of the $\mu_i$ is the smallest eigenvalue of $H$, hence we may just select the smallest indices $i_0$ such that it is the case and consider the associated vector. We can then iterate by applying the same procedure to $(I-u_{i_0}u_{i_0}^T)H(I-u_{i_0}u_{i_0}^T) + (|H|+1) u_{i_0} u_{i_0}^T$, and similarly by induction. We conclude with the fact that $x\mapsto\na^2U(x)$ is continuous, hence measurable.
\end{proof}

As in $\R^d$, for $\varepsilon>0$ the overdamped Langevin process is defined on $\T^d$ as the solution of \eqref{eq:Langevin}. It is ergodic with respect to $\mu_\varepsilon$ the probability measure with density proportional to $e^{-U/\varepsilon}$ on the torus.  Similarly, we define the noisy ISD  as the solution on $\T^d$ of
\begin{equation}\label{eq:noisy-ISD}
    \dd X_t = -\left(I-2v_1(X_t)v_1(X_t)^T\right)\na U(X_t)\dd t + \sqrt{2\varepsilon}\dd B_t\,,
\end{equation}
which, contrary to its deterministic counterpart, is always well-defined (see Proposition~\ref{prop:ISD-theory} below). Concerning the noisy version of GAD, we may wonder whether it is necessary to add noise to both coordinates, and thus in general we can consider the noisy GAD as the solution on $\T^d\times \mathbb S^{d-1}$ of
\begin{equation}\label{eq:noisy-GAD}
    \left\{ \begin{array}{lll}
    \dd X_t &=& -\left(I-2V_tV_t^T\right)\na U(X_t)\dd t + \sqrt{2\varepsilon}\dd B_t, \\
    \dd V_t &=&  (I-V_tV_t^T)\co -\frac{1}\eta \na^2U(X_t)V_t\dd t + \sqrt{2\varepsilon'} \dd B_t' \cf - d \varepsilon' \dd t, \end{array}\right.
\end{equation}
where $\eta>0$, $\varepsilon'\geqslant 0$ and $B,B'$ are independent $d$-dimensional Brownian motions. The second line is such that if $\na^2 U(X_t)$ is constant then this is an overdamped Langevin diffusion on the sphere with potential given by the Rayleigh quotient (see e.g. \cite[Section 3.2.3]{LelievreFreeEnergy} or \cite{Alfonsi}). In particular, according to \cite[Lemma 3.18]{LelievreFreeEnergy}, $V_t\in\mathbb S^{d-1}$ for all $t\geqslant 0$. In fact, by contrast with the deterministic case, since the noisy ISD is well defined, we will not insist much on the noisy GAD, and simply discuss the case $\varepsilon'=0$ in Section~\ref{s-sec:ISD-theory}.

Next, we define the idealised switched process (ISP) as the Markov process $(X_t,I_t)_{t\geqslant 0}$ on $\T^d\times\{0,1\}$ where $(I_t)_{t\geqslant 0}$ is a Poisson process on $\{0,1\}$ with jump rate $\nu>0$ and $X$ solves
\begin{equation}\label{def:diffswitch}
    \dd X_t \ = \   H_{I_t}(X_t)\dd t + \sqrt{2\varepsilon}\dd B_t,
\end{equation}
with  $(B_t)_{t\geqslant 0}$ a Brownian motion independent from  $(I_t)_{t\geqslant 0}$ and, for all $x\in\mathbb T^d$,
\[H_0(x)= -\na U(x)\,,\qquad H_1(x)=-(I-2v_1(x)v_1(x)^T)\na U(x)\,.\]
In other words, the process alternates between the overdamped Langevin dynamics and the noisy ISD, the time being two switching events being independent and distributed according to an exponential law of parameter $\nu$. Likewise, we can define the gentle switched process (GSP) as the Markov process $(X_t,V_t,I_t)_{t\geqslant 0}$ on $\T^d\times\mathbb S^{d-1}\times\{0,1\}$ where
\begin{equation}\label{eq:noisy-GAD-switch}
    \left\{ \begin{array}{lll}
    \dd X_t &=& \overline{H}_{I_t}(X_t,V_t)+ \sqrt{2\varepsilon}\dd B_t, \\
    \dd V_t &=&(I-V_tV_t^T)\co -\frac{1}\eta \na^2U(X_t)V_t\dd t + \sqrt{2\varepsilon'} \dd B_t' \cf - d \varepsilon' \dd t, \end{array}\right.
\end{equation}
where $(I_t)_{t\geqslant 0}$ is as in the ISP, independent from $B,B'$, and
\[\overline{H}_0(x,v)= -\na U(x)\,,\qquad \overline{H}_1(x,v)=-(I-2vv^T)\na U(x)\,.\]

Finally, let us introduce slight variations of the processes defined above. Consider
\[
\mathcal S = \left\{x\in\T^d, \lambda_1(x)=\lambda_2(x) \right\}\,,
\]
which is the set of singularities of the deterministic ISD flow. As pointed out in \cite{ISD-GAD}, some points of $\mathcal S$ may be attracting singularities for this flow, and thus they will also locally attract the noisy ISD when $\varepsilon $ is small, or the noisy GAD when $\varepsilon$ and $\eta$ are small.
 This behaviour may be mitigated as follows. Fix some non-decreasing $f:\R_+\to [1,2]$ such that $f(0)=1$ and $f(r)=2$ for all $r$ larger than some small threshold $r_*>0$. Then, setting $a(x) = f(\lambda_2(x)-\lambda_1(x))$ the drift $H_1$ in the noisy ISD may be replaced by 
 \[\hat H_1(x) = - \po I - a(x) v_1(x)v_1(x)^T - (2-a(x)) v_2(x)v_2(x)^T\pf \na U(x)\,.  \]
 In other words, when $\lambda_1(x)$ and $\lambda_2(x)$ are clearly distinguished, then we recover the previous noisy ISD, however, on $\mathcal S$, rather than undergoing  an orthogonal reflection with respect to $v_1(x)$, $\nabla U(x)$ is orthogonally projected on the orthogonal of the space spanned by $v_1(x),v_2(x)$, which means that, in these two directions, the process behave like a Brownian motion. Natural choices for $f$ would be piecewise constant (with $f(r)=1$ for $r<r_*$) or piecewise linear (with $f(r)=1+r/r_*$ for $r<r_*$). In the second case, the drift $\bar H_1$ is continuous.
 
An analogous modification can be applied to the drift $\overline{H}_1$ of the noisy GAD, in which case the process is then $(X,V_1,V_2)$ where $V_1,V_2$ are orthogonal vectors of $\mathbb S^{d-1}$ whose evolution is given, in the case $\varepsilon'=0$, by 
 \begin{align*}
   \dd V_{1,t} &= -\frac{1}\eta \po I-V_{1,t}V_{1,t}^T  \pf \na^2U(X_t)V_{1,t}\dd t\\
   \dd V_{2,t} & = -\frac{1}\eta \po I-V_{2,t}V_{2,t}^T - 2 V_{1,t}V_{1,t}^T  \pf \na^2U(X_t)V_{2,t}\dd t\,.
 \end{align*}
 (we refer to \cite[Section 3.2.3]{LelievreFreeEnergy} to derive the case $\varepsilon'>0$, where the constraints are now that $|v_1|^2=|v_2|^2=1$ and $v_1^T v_2 = 0$).
 
 This variation can be generalised to any degree of degeneracy of the smallest eigenvalue, i.e. for $k\geqslant 2$ modifying the drift $H_1$ in the vicinity of the set where $\lambda_1(x)=\dots=\lambda_k(x)$  (generically these sets are empty for $k>2$, but in many practical cases they are not, due to some symmetries in the problem).



\subsection{Related works}\label{s-sec:biblio}

The overdamped and underdamped Langevin process have been widely studied for decades. For an introduction in the context of stochastic algorithm, see \cite{LelievreFreeEnergy,lelievre2016partial}. In an optimisation rather than sampling context, it is also possible to use a temperature $\varepsilon$ that varies with time. For example, the simulated annealing uses a temperature that goes to $0$, so that the process converges in probability towards the minima of $U$, see  the classical \cite{Holley} or the more recent  \cite{kinetic-annealing} for more references.



Concerning saddle point search algorithms, we refer to the works already mentioned in Section~\ref{s-sec:intro-dertemini}, namely \cite{henkelman2002methods,E_2002} for minimal path between two known local minima and \cite{Non-conservative-GAD,high-index-GAD,minimisation-formulation,ISD-GAD,Weiman2012,Bofill2015,Gu2018} for  eigenvector-following type algorithms like ISD or GAD. One could also see a saddle point as a minimum of the function $x\mapsto|\na U(x)|^2$ or variants, which is done in \cite{duncan2014biased,Bonfanti2017}.   


Although the use of switched diffusion processes for stochastic optimization algorithms and the specific processes introduced in Section~\ref{s-sec:def-process} are new, there are already many works on switched (also called Markov-modulated) diffusion processes, usually in contexts and with motivations rather different from ours. In the case of an Ornstein-Uhlenbeck process with switched parameter, it is possible to say much about the long time behaviour of the process, see \cite{Saporta05,Bardet09,Lindskog20}. In a more general case, the fast switching and small temperature regime has been studied in \cite{guillin03,HUANG16,Nguyen21}. Switched process are also used for scientific modelling, see \cite{finance} for an example in finance.

\section{Theoretical analysis}\label{sec:Theory}

Before turning to numerical experiment, we are interested in the existence of the different processes, and some of their properties, mainly in the long time limit. For the well-known Langevin process,  we refer to \cite{LelievreFreeEnergy,lelievre2016partial}. 

A valuable tool to study long time behavior of a process is the Doeblin condition. A Markov kernel $\mathcal P:\T^d\to\mathcal M^1(\T^d)$ is said to satisfy a Doeblin condition if there exist $c>0$ and a probability measure $\mu$ such that for all $x\in\T^d$, $\mathcal P(x) \geqslant c\mu$. This implies that $(\mathcal P^n)_n$ converges in total variation to the unique probability measure $\mu_{\infty}$ such that $\mu_{\infty}\mathcal P=\mu_{\infty}$, see for example~\cite{HARRIS}. In the continuous-time setting, this can be written as follow: there exists $t_0,c>0$ and a probability measure $\mu$, such that $\delta_xP_{t_0}\geqslant c\mu$ for all $x\in\T^d$, with $P_t$ the semi-group of the process, and $\delta_xP_{t_0}$ is hence the law of the process starting from $x$ a time $t_0$.

\subsection{Noisy ISD and Noisy GAD}\label{s-sec:ISD-theory}




Since we work on the compact torus and the drift of the noisy ISD \eqref{eq:noisy-ISD} is measurable and bounded, the well-posedness and long-time behavior follows from classical results on diffusion processes. We summarize these results on the next proposition. Apart from the continuous-time equation \eqref{eq:noisy-ISD}, we are also interested in the corresponding  Euler-Maruyama scheme. The  noisy ISD is covered by the following:

\begin{prop}\label{prop:ISD-theory}
Consider $\sigma>0$, $b:\T^d\to \R^d$ measurable and bounded, and $X$ solution to
\begin{equation}\label{eq:diffusion-elliptique}
\dd X_t = b(X_t)\dd t + \sigma \dd B_t.
\end{equation}
For any initial condition $x_0\in \T^d$, and Brownian motion $B$, there exists a unique Markov process $X^{x_0}$ solution to equation~\eqref{eq:diffusion-elliptique}, with Brownian motion $B$ and initial condition $x_0$.

Moreover, this process admits a unique invariant probability measure which admits a density with respect to the Lebesgue measure, and the law of the process converges exponentially fast towards this stationary measure in the total variation distance.

The Euler-Maruyama scheme defined for $\delta>0$ by
\[
X_{n+1} = X_{n} + \delta b(X_n) + \sqrt{\delta} G_n
\]with $(G_n)_n$ a sequence of standard Gaussian variables, admits as well a unique stationary measure $\mu^\delta$, and there exists $c,C,\delta_0>0$ independent from $\delta$ such that for all $0<\delta<\delta_0$, $x_0\in\T^d$:
\[
\|\mathcal Law(X_n) - \mu^\delta\|_{TV}\leqslant Ce^{-c\delta n}.
\]
\end{prop}

\begin{proof}
The existence of a unique solution $X^{x_0}$ to Equation~\eqref{eq:diffusion-elliptique} with initial condition $x_0$ comes from~\cite[Corrolary 7.1.7 and 8.1.7]{Strook-Varadhan}. By It\^o's formula, the law $M^{x_0}$ of $(X^{x_0}_t)_{t\in[0,T]}$ (a probability measure on $\C([0,T],\T^d)$) is a measure solution to the Kolmogorov equation
\[
\partial_t M^{x_0} = L^*M^{x_0}, 
\]
where $L^*$ is the dual in $L^2(\R^d)$ of the generator 
\[L\varphi = b(x)\cdot\na \varphi + \frac{\sigma^2}{2}\Delta \varphi,\]in the sense that for all function $\varphi\in\C^2([0,T]\times\T^d, \R)$
\[\int_{[0,T]\times\T^d}(\partial_t - L)\varphi(t,x) M^{x_0}(\dd t,\dd x) = 0 .\]
Thus, according to \cite[Proposition 6.5.1]{Krylov}, $\mathcal Law(X_t^{x_0})$ (the time-marginal of $M^{x_0}$) admits a density $h^{x_0}(t,\cdot)$ with respect to the Lebesgue measure, for all $t>0$, and for any compact interval $J\subset\R_+$ \[
h^{x_0}\in \mathbb H^{p,1}(J) := \left\{ u:\R_+\times\T^d\to \R, \int_J \|u(t,\cdot)\|_{H^{p,1}} \dd t <\infty \right\},
\]
where $\|\cdot\|_{H^{p,1}}$ is the classical Sobolev norm, for some $p>d+2$. Since we are in a compact set, $h^{x_0}\in \mathbb H^{2,1}$, and \cite[Proposition 6.2.7]{Krylov} yields that $h^{x_0}$ satisfies an Harnack inequality for $t$ great enough: there exists $C,\tau>0$ such that for all $x_0\in\T^d$ and $t$ great enough:
\[
\sup_{\T^d}h^{x_0}(t-\tau,\cdot) \leqslant C\inf_{\T^d}h^{x_0}(t,\cdot).
\]
Since $h^{x_0}(t-\tau,\cdot)$ is a probability density, we have $\sup_{\T^d}h^{x_0}(t-\tau,\cdot)\geqslant \mathrm{Leb}(\T^d)$, where $\mathrm{Leb}$ stands for the Lebesgue measure. In particular, there exists $t_0>0$ such that:
\[
\inf_{x_0\in\T^d}\inf_{\T^d}h^{x_0}(t_0,\cdot) >0.
\]
This is a Doeblin condition for the law of the process, with reference measure the Lebesgue measure, and in particular it classically implies the existence of a unique stationary measure for the process~\eqref{eq:noisy-ISD} and the exponential convergence of the law at time $t$ toward this equilibrium.

For the Euler-Maruyama scheme, fix $n=\lfloor t/\delta \rfloor$, and write:
\[
X_n^{x_0} = x_0 + \delta \sum_{k=0}^{n-1} b(X_k^{x_0}) + \sqrt{\delta}  \sum_{k=0}^{n-1} G_k, 
\]
and 
\[
G = n^{-1/2}\sum_{k=0}^{n-1} G_k.
\]
We have $\delta \sum_{k=0}^{n-1} b(X_k^{x_0}) \leqslant t \|b\|_{\infty}$. If $A\subset \T^d$ is a measurable set, we are looking for a lower bound on $\mathbb P(X_n^{x_0}\in A)$. Let $\bar X_n^{x_0}$ be the Euler scheme seen on $\R^d$ rather than $\T^d$ (considering $b$ as a periodic function), with initial condition given as the representant of $x_0$ in $[0,1)^d$ and let $\bar A$ be the intersection of $[0,1)^d$ with the pre-image of $A$ by the periodic quotient $\R^d\rightarrow \T^d$. Then, $\mathbb P(X_n^{x_0}\in A) \geqslant \mathbb P(\bar X_n^{x_0}\in \bar A)$, and
\begin{align*}
    \mathbb P(\bar X_n^{x_0}\in \bar A) &= \mathbb P\left(\sqrt{n\delta}G \in \bar A - x_0 -  \delta \sum_{k=0}^{n-1} b(X_k) \right)\\ &\geqslant \inf_{|a|\leqslant 2\pi d + t\|b\|_{\infty}} \int_{\bar A- a}e^{-|y|^2/(t-\delta)}\dd y \\ & \geqslant e^{-2(4\pi d + t\|b\|_{\infty})/t}\mathrm{Leb}(A).
\end{align*}
In other words, the $n$-step transition kernel of the Euler scheme satisfies a Doeblin condition with constant independent from the step size $\delta$. This implies the existence of a unique stationary measure $\mu^{\delta}$ such that for all $k\in\mathbb{N}$,
\[
\|\mathcal Law(X_{k\lfloor t/\delta \rfloor}) - \mu^\delta\|_{TV}\leqslant \alpha^k\|\mathcal Law(X_0) - \mu^\delta\|_{TV}.
\]
This gives for all $n\in\mathbb N$
\begin{align*}
\|\mathcal Law(X_n) - \mu^\delta\|_{TV} 
&\leqslant \|\mathcal Law(X_{\lfloor t/\delta \rfloor\lfloor n/\lfloor t/\delta \rfloor \rfloor}) - \mu^\delta\|_{TV} \\ 
&\leqslant \alpha^{\lfloor n/\lfloor t/\delta \rfloor \rfloor}\|\mathcal Law(X_0) - \mu^\delta\|_{TV} \leqslant Ce^{-c\delta n},    
\end{align*}
which concludes the proof.
\end{proof}


We now turn to the analysis of the noisy GAD \eqref{eq:noisy-GAD}. When $\varepsilon'>0$, we get an elliptic diffusion with bounded drift on a smooth compact manifold with no boundary, so essentially the adaptation of Proposition~\ref{prop:ISD-theory} in this more general context apply. In the following we focus on the case $\varepsilon'=0$. Indeed, one may think that it is only necessary to have noise on the position in order to visit the whole space, and then the auxiliary vector $V$ will go to the associated eigenvectors. The question is whether the process is hypoelliptic and controllable. However, it is clear that it cannot be the case in general, simply by considering  counter-examples of the form
\[
U(x) = \sum_{i=1}^d U_i(x_i)\,.
\]
Indeed, in these cases, if e.g.  $v_0=(1,0,...,0)$, then $v_t=v_0$ for all $t\geqslant 0$, hence there are no hope of having a unique stationary measure, and convergences towards it for all initial conditions. Let us discuss some some additional  conditions that ensure this  result. First, notice that if $(X_t,V_t)_{t\geqslant 0}$ is a noisy GAD then so is $(X_t,-V_t)_{t\geqslant 0}$. In fact the relevant variable to define the process is not $V_t$ but the class $\{V_t,-V_t\}$ on the projective sphere $\mathbb P^d = \mathbb S^{d-1}/\mathcal R$ with  the equivalence relation given by $v\mathcal R u$ iff $u=-v$ or $u=v$, and we identify the process with its projection on $\mathbb T^d\times \mathbb P^d$ (which is still a Markov process, solution of the same SDE \eqref{eq:noisy-GAD} but where the drift of $V_t$ is seen as a vector field on $\mathbb P^d$).

\begin{prop}\label{prop:hypocoer}
If $U$ is $\mathcal C^3$, then the noisy GAD is well defined. If in addition, we suppose that the potential $U$ satisfies the following:
\begin{itemize}
    \item There exists $(\tilde x,\tilde v)\in\T^d\times\mathbb P^{d}$ such that:
    \[\po \partial_{x_i} \na^2U(\tilde x)\tilde v \pf_{1\leqslant i\leqslant d}\]spans $\tilde v^\perp$. 
    \item For all $v\in \mathbb P^{d}$, there exists $H$ in the convex hull of the set
    \[
    \left\{ \na^2U(x), x\in \T^d \right\}
    \]such that, denoting by $(\tilde \lambda_i)$ the ordered eigenvalues of $H$ and $(\tilde v_i)$ the associated eigenvectors,
    \[    \left\{ \begin{array}{lll}
    v &=& \tilde v_1, \\
    \tilde \lambda_1 &<& \tilde \lambda_2. \end{array}\right.\]
Then the noisy GAD admits a unique stationary measure $\mu_{GAD}$ on $\mathbb T^d\times \mathbb P^{d}$ with a positive density with respect to the Lebesgue measure. Moreover, for any initial condition, the law at time $t$ of the process converges exponentially fast to $\mu_{GAD}$ in total variation.
\end{itemize}
\end{prop}

\begin{rem}
The second condition is implied by: For all $v\in \mathbb P^{d}$, there exists $x\in\T^d$ such that
    \[    \left\{ \begin{array}{lll}
    v &=& v_1(x), \\
    \lambda_1(x) &<& \lambda_2(x). \end{array}\right.\]
\end{rem}

\begin{proof}
If $U$ is $\mathcal C^3$, then the noisy GAD is an SDE with Lipschitz coefficient, hence its well-posedness.
The first additional assumption ensures that the generator of the noisy GAD
\[
\mathcal L = -(I-2vv^t)\na U(x)\cdot \na_x + \varepsilon \Delta_x - \frac{1}{\eta}(I-vv^t)\na^2U(x) v \cdot\na_v
\]
satisfies a weak H{\"o}rmander condition at $(\tilde x,\tilde v)$, see~\cite{controllability}. 

Next, we want to establish the controllability of the process, in the following sense: fix $( x_{ini}, v_{ini}),\ ( x_f, v_f)\in \mathbb T^d\times \mathbb P^{d}$, $\delta>0$. We want to show that there exists $T>0$ and a control $u\in \mathcal C^0([0,T],\R^d)$ such that the solution of 
\begin{equation}\label{eq:control}
    \left\{ \begin{array}{lll}
    x'(t) &=& -\left(I-2v(t)v(t)^T\right)\na U(x(t)) + u(t), \\
    v'(t) &=& -\frac{1}\eta (I-v(t)v(t)^T)\na^2U(x(t))v(t), \end{array}\right.
\end{equation}
with $(x(0),v(0))=( x_{ini}, v_{ini})$, satisfies $\left\| (x(1),v(1))-( x_f, v_f) \right\| \leqslant \delta$. 

We first suppose that $v_{ini}$ and $v_f$ are not orthogonal. Fix $H$ such that the second condition of Proposition~\ref{prop:hypocoer} is satisfied, and $x_1,\dots,x_p$, and $\alpha_1,\dots,\alpha_p$ such that
\[
H = \sum_{i=1}^p \alpha_i\na^2U(x_i),\quad \sum_{i=1}^p \alpha_i = 1.
\]
Fix $T>0$, $n\in\N$, and $\gamma>0$ such that $T=n(\gamma+\gamma^2 + 2\gamma^4) + \gamma^2 + \gamma^4$. Write $\beta_i = \alpha_1 + ... + \alpha_{i}$, and for $k\in \llbracket 0,n \rrbracket$:
\begin{equation*}
    u (t) = \left\{ \begin{array}{lll}
    (-x_{ini} +  x_1)/(\alpha_1\gamma^2) &\text{ if }& 0< t < \alpha_1\gamma^2, \\
    (I-2v(t)v(t)^t)\na U(x_1) &\text{ if }& \alpha_1(\gamma^2+\gamma^4) < t < \alpha_1(\gamma+\gamma^2 + \gamma^4), \\
    ( x_{i+1} - x_i)/(\alpha_{i+1}\gamma^2) &\text{ if }& (k+\beta_i)(\gamma + \gamma^2 + 2\gamma^4) \\ & & \qquad\qquad < t < (k+\beta_i)(\gamma + \gamma^2 + 2\gamma^4) + \alpha_{i+1}\gamma^2,\\
    (I-2v(t)v(t)^t)\na U(x_i) &\text{ if }& (k+\beta_i)(\gamma + \gamma^2 + 2\gamma^4) + \alpha_{i+1}(\gamma^2 +\gamma^4) \\ & & \qquad \qquad< t < (k+\beta_{i+1})(\gamma + \gamma^2 + 2\gamma^4) - \alpha_{i+1}\gamma^4,\\
    (x_f -  x_p)/\gamma^2 &\text{ if }& T -\gamma^2  < t <T, \end{array}\right.
\end{equation*}
and $u$ is linear on the remaining intervals of size $\gamma^4$. Write as well 
\begin{equation*}
    \bar x (t) = x_i\text{ for } (k+\beta_{i-1}) \gamma < t < (k+\beta_{i}) \gamma,\ k\in\N.
\end{equation*}
As $\gamma$ goes to 0, $v$ converges towards the solution to
\[
\bar v'(t) = -\frac{1}\eta (I-\bar v(t)\bar v(t)^T)\na^2U(\bar x(t))\bar v(t)
\]
which is itself close to 
\[
w'(t) = -\frac{1}\eta (I-w(t)w(t)^T)Hw(t).
\]
For clarity, we only show the second point which is the more complex one. We may write for all $0\leqslant u \leqslant T$:
\begin{align*}
    &\left\|\int_0^u (I-\bar v_s\bar v_s^t)\na^2U(\bar x)\bar v_s \dd s - \int_0^u (I-\bar v_s\bar v_s^t)H\bar v_s \dd s\right\| \\ &\quad \leqslant \left\|\int_0^u (I-\bar v_s\bar v_s^t)\na^2U(\bar x)\bar v_s \dd s - \sum_{k=0}^{m-1}\sum_{i=1}^p(I-\bar v_{k\gamma}\bar v_{k\gamma}^t)\na^2U(x_i)\bar v_{k\gamma}\alpha_i\gamma\right\| \\ &\quad  +\left\| \sum_{k=0}^{m-1}\sum_{i=1}^p (I-\bar v_{k\gamma}\bar v_{k\gamma}^t)\na^2U(x_i)\bar v_{k\gamma}\alpha_i\gamma - \sum_{k=0}^{m-1}\sum_{i=1}^p(I-\bar v_{k\gamma}\bar v_{k\gamma}^t)H\bar v_{k\gamma}\gamma \right\| \\ &\quad + \left\|\sum_{k=0}^{m-1}\sum_{i=1}^p(I-\bar v_{k\gamma}\bar v_{k\gamma}^t)H\bar v_{k\gamma}\gamma - \int_0^u (I-\bar v_s\bar v_s^t)H\bar v_s \dd s \right\|.
\end{align*}
where $m=\lfloor u/\gamma \rfloor$. By hypothesis on $H$, the second term is $0$. $\bar v$ depends on $\gamma$ by definition, but it is $\|\na U\|_{\infty}$-Lipschitz for all $\gamma>0$, hence the first and third term are going to $0$ as $\gamma\rightarrow 0$, uniformly on $0\leqslant u \leqslant T$. The map $v\mapsto (I-vv^t)Hv$ is $\mathcal C^1$ on a compact set, hence is Lipschitz. Thus we have:
\begin{equation*}
    \|\bar v(u) - w(u) \| \leqslant L\int_0^u \sup_{r<s}\|\bar v(r) - w(r) \| \dd s + \left\|\int_0^u (I-\bar v_s\bar v_s^t)\na^2U(\bar x)\bar v_s \dd s - \int_0^u (I-\bar v_s\bar v_s^t)H\bar v_s \dd s\right\|,
\end{equation*}
and
\[
\sup_{0\leqslant u\leqslant} \|\bar v(u) - w(u) \| \leqslant e^{LT}\sup_{0\leqslant u\leqslant}\left\|\int_0^u (I-\bar v_s\bar v_s^t)\na^2U(\bar x)\bar v_s \dd s - \int_0^u (I-\bar v_s\bar v_s^t)H\bar v_s \dd s\right\| \underset{\gamma\rightarrow 0}{\rightarrow}0.
\]
Now, all it remains to do is to fix $T>0$ great enough so that $w(T)$ is in the $\delta$-neighborhood of $v_f$, then $\gamma$ small enough so that $v(T)$ is as well, and $n$ accordingly.

If $v_{ini}$ and $v_f$ are orthogonal, writing $\pi$ for the canonical projection from  $\mathbb S^{d-1}$ to $\mathbb P^{d}$, there exists an intermediate $v_{inter}$ which is not orthogonal to neither  $v_{ini}$ or $v_f$, and the same control allows to go from $(x_{ini},v_{ini})$ to $(x_{ini},v_{inter})$ in a time $T$, and from $(x_{ini},v_{inter})$ to $(x_{f},v_{f})$ in another time $T$.

We conclude with \cite[Theorem 2.1]{controllability} for the existence of the stationary measure and the convergence, and with \cite[Theorem 5.2]{density} for the positive density.
\end{proof}

\subsection{Switched  processes }\label{s-sec:Switched-theory}

When the mode $(I_t)_{t\geqslant 0}$ is an autonomous Markov chain on a finite set $\cco 0,m-1\ccf$, $m\in\N$ (i.e., as it is the case in the processes we have considered in Section~\ref{s-sec:def-process}, if its jump rate does not depend on $X$) then the well-posedness of switched processes of the form
\begin{equation}\label{eq:diffusion-gen}
\dd   X_t = H_{I_t}(X_t)\dd t + \sigma_{I_t}(X_t)\dd B_t,
\end{equation}
with some drifts $H_0,\dots, H_{m-1}$ immediately follows from the well-posedness of the corresponding diffusion processes for each $i\in\cco 0,m-1\ccf$, since the process is then simply defined by induction along the jump times of $I$. In particular, in view of the previous section, the ISP \eqref{def:diffswitch} is well defined. 

Moreover, in general, it is sufficient that one of the diffusion processes satisfies a Doeblin condition to imply the same for the switched process: 




\begin{prop}\label{prop:swith-theory}

If $I$ is irreducible and there exists $i_0$ such that the diffusion~\eqref{eq:diffusion-gen} with $I_t=i_0$  satisfies a Doeblin condition (see the introduction of Section~\ref{sec:Theory}), then the switched process admits a unique invariant probability measure, and the law of the switched process converges exponentially fast towards this stationary measure.
\end{prop}

\begin{proof}
The existence results from the previous construction.
Fix $t_0>0$ such that for all $t>t_0$, $x\in\R^d$, $A\subset\T^d$ measurable, 
\[
\mathbb P_x(\bar X_t \in A) \geqslant cl(A)
\]
where $c>0$ and $\bar X$ is a solution to equation~\eqref{eq:diffusion-gen} with $I_0=i_0$. Since $I$ is an irreducible Markov process, we have for $t>0$
\[
\mathbb P(I_s = i_0,\ \forall s\in [t,2t]) >0.
\]
Hence we have:
\begin{align*}
    \mathbb P(X_{2t}\in A, I_{2t} = i_0) &\geqslant \mathbb P(X_{2t}\in A, I_s = i_0,\ \forall s\in [t,2t]) \\ & =\E\po\mathbb P_{X_t}\po \bar X_t \in A \pf \mathbbm{1}_{I_s=i_0, \forall s\in[t,2t]} \pf \\ &\geqslant cl(A) \mathbb P(I_s = i_0,\ \forall s\in [t,2t]) \\ &= \tilde c l(A),
\end{align*}
which is a Doeblin condition with reference measure $l\otimes \delta_{i_0}$, and this concludes the proof.
\end{proof}

\begin{cor}
The ISP is well-defined, admit stationary measure, and its laws converge as $t$ goes to infinity exponentially fast towards its unique stationary measure in the total variation distance.
\end{cor}

\section{Numerical experiments}\label{sec:Num}

The main goal of the present work is to gain some empirical insights on the qualitative behavior of the processes introduced in Section~\ref{s-sec:def-process} on simple models. Some questions that we have in particular are the following: what happens if the direction of the gentlest ascend starting from a local minimum does not correspond to the direction the process has to take  eventually to find a saddle point? How does the process behave in front of attractive singularities of the ISD flow? How does the efficiency of the process (in terms of exploration) depend on the switching rate? How does it compare to a basic overdamped Langevin process?

We will consider three models. The first one is a mixture of Gaussian, for which we can easily chose where are the minima and the saddle point and what are the gentlest ascend directions at the minima.  The second one is the function studied in~\cite{ISD-GAD}, for which the unique saddle point is separated from the two minima by singular lines. Finally, the third example, in higher dimension and closer to a genuine application, is the Lennard-Johns cluster model with $7$ particles.

Contrary to the theoretical analysis, most processes here live on $\R^d$, which may raise a stability issue. This can be seen in dimension $1$, where the process switch between a gradient descent and a gradient ascent, and  thus can be non recurrent, or even  explosive e.g. for potential of the form $|x|^4$ at infinity. This can be solved by a suitable modification of the dynamics outside some compact set, for instance replacing $H_1$ in \eqref{def:diffswitch} by $\tilde H_1(x) = H_1(x) 1_{|x|\leqslant R} + H_0(x) 1_{|x|>R}$ or a smooth interpolation, or adding a deterministic jump from $I_t=1$ to $I_t=0$ when $|X_t|\geqslant R$. Then, classical Lyapunov conditions on $U$ implies that the process remains stable. However, this will not be necessary in our simple numerical experiments.

\subsection{First 2D model: a mixture of Gaussians}\label{s-sec:GAD-num-1}

We consider here a mixture of Gaussian in dimension 2. The potential is of the form:
\[
U_1(x,y) = -\ln{\left(\frac{1}{2}e^{-(x^2+y^2)} + \frac{1}{2}e^{-s_x(x-m_x)^2-s_y(y-m_y)^2}\right)},
\]
with some parameters $m_x,m_y\in \R$, $s_x,s_y>0$. This is a classical model in statistics for multi-modal problems. Here we take $(m_x,m_y)=(4,0)$ and consider two cases, either $(s_x,s_y)=(3,1)$ or $(s_x,s_y)=(1,3)$, see Figure~\ref{fig:evolution-(3,1)} and~\ref{fig:evolution-(1,3)}. In those figures, we represent a typical trajectory of the noisy ISD and of the ISP, with $\varepsilon=0.03$, the final time $T=10$, and $\nu = 0.2$ in the case of the ISP, as well as the level lines of the potential. In these two cases, there are two local minimizers $(\tilde x_1,0)$ and $(\tilde x_2,0)$ separated by a unique saddle point $(\tilde z,0)$, where $\tilde x_1<\tilde z < \tilde x_2$, but they differ by the direction taken by the saddle search in the well from the left. The yellow part represent the trajectories of the noisy ISD, or the part of the ISP where $I_t=1$, and in blue the part of the ISP where $I_t=0$ (corresponding to a gradient descent).

\begin{figure}
\begin{subfigure}{.5\textwidth}
  \centering
  \includegraphics[width=.9\linewidth]{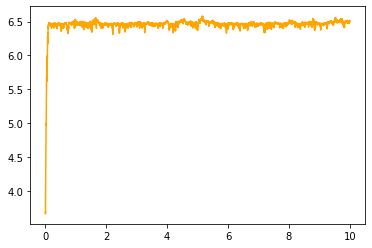}
  \caption{Evolution of $U$ along the noisy ISD.}
  \label{sfig:ener-ISD(3,1)}
\end{subfigure}%
\begin{subfigure}{.5\textwidth}
  \centering
  \includegraphics[width=.9\linewidth]{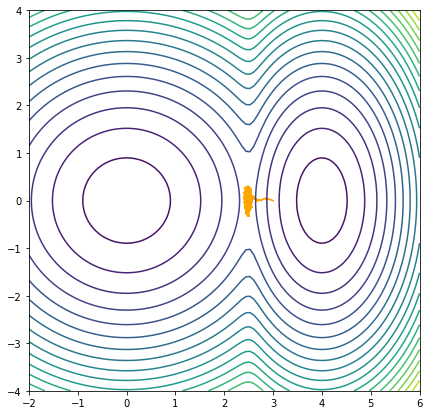}
  \caption{Trajectory of the noisy ISD.}
  \label{sfig:traj-ISD(3,1)}
\end{subfigure}
\\
\begin{subfigure}{.5\textwidth}
  \centering
  \includegraphics[width=.9\linewidth]{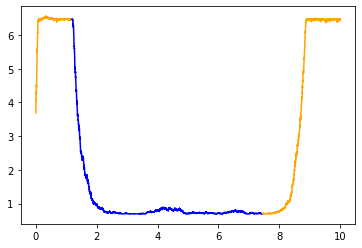}
  \caption{Evolution of $U$ along the switched process.}
  \label{sfig:ener-SWI(3,1)}
\end{subfigure}%
\begin{subfigure}{.5\textwidth}
  \centering
  \includegraphics[width=.9\linewidth]{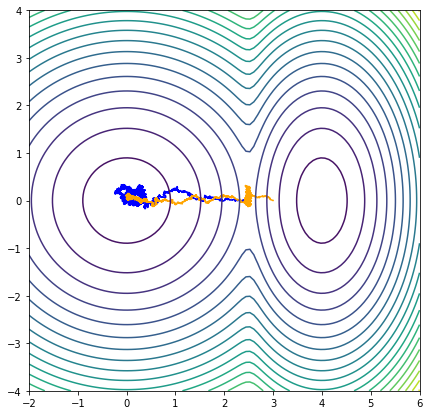}
  \caption{Trajectory of the switched process.}
  \label{sfig:traj-SWI(3,1)}
\end{subfigure}
\caption{Case $(s_x,s_y)=(3,1)$ }
\label{fig:evolution-(3,1)}
\end{figure}

\begin{figure}
\begin{subfigure}{.5\textwidth}
  \centering
  \includegraphics[width=.9\linewidth]{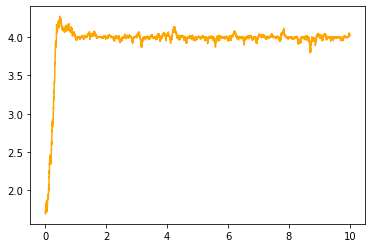}
  \caption{Evolution of $U$ along the noisy ISD.}
  \label{sfig:ener-ISD(1,3)}
\end{subfigure}%
\begin{subfigure}{.5\textwidth}
  \centering
  \includegraphics[width=.9\linewidth]{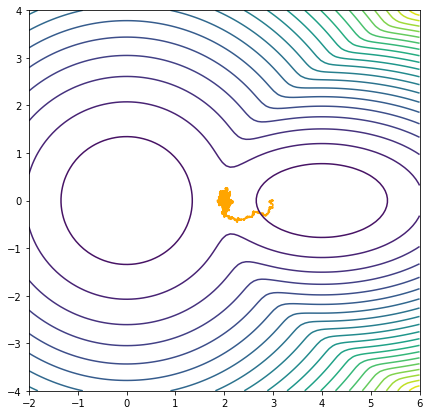}
  \caption{Trajectory of the noisy ISD.}
  \label{sfig:traj-ISD(1,3)}
\end{subfigure}
\\
\begin{subfigure}{.5\textwidth}
  \centering
  \includegraphics[width=.9\linewidth]{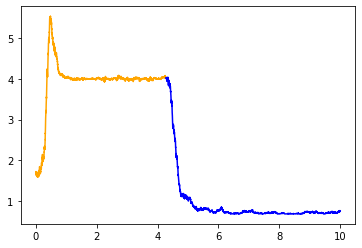}
  \caption{Evolution of $U$ along the switched process.}
  \label{sfig:ener-SWI(1,3)}
\end{subfigure}%
\begin{subfigure}{.5\textwidth}
  \centering
  \includegraphics[width=.9\linewidth]{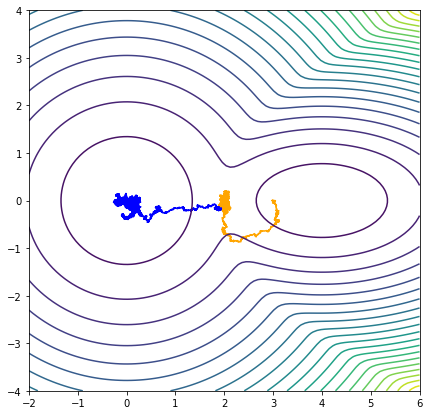}
  \caption{Trajectory of the switched process.}
  \label{sfig:traj-SWI(1,3)}
\end{subfigure}
\caption{Case $(s_x,s_y)=(1,3)$.}
\label{fig:evolution-(1,3)}
\end{figure}


According to the Large Deviation Principle, for a small $\varepsilon$,  if the initial condition corresponds to the right local minima $(\tilde x_2,0)$, the overdamped Langevin  process \eqref{eq:Langevin} will typically stays in the vicinity of this minimizer for a time exponentially long with $\varepsilon$, and will likely go from there to the other minimizer by following in reverse the trajectory of a gradient descent starting from a point arbitrarily close to the saddle point (which in the present case is a  straight line from the minimum to the saddle point), see~\cite[Chapter 4, Theorem 2.1]{Freidlin-Wentzell} as well as their computation of the quasi-potential.

In the case $(s_x,s_y)=(3,1)$, with some initial condition $(x_0,0)$, $\tilde z<x_0<\tilde x_2$, the noisy (and deterministic) ISD will follow the same reactive trajectory as the overdamped Langevin process, the gentlest way to go to the saddle point being here along the gradient, but it happens much earlier (since this is the behaviour of the deterministic system at $\varepsilon=0$ and not a large deviation from it). However, in the case, $(s_x,s_y)=(1,3)$, the gentlest way to leave the right minima is to start along the directions $(0,1)$ and $(0,-1)$, see figure~\ref{fig:evolution-(1,3)}. Nevertheless, independently from this difference in the beginning of the trajectory, in both cases the saddle point is found by the deterministic dynamics for initial condition between $\tilde x_1$ and $\tilde x_2$. 

For initial condition $(x_0,y_0)$ with $x_0>\tilde x_2$, the ISD goes to infinity. Hence, the dynamics has to be modified outside some compact set, as explained in the introduction of this section, so that the process does not escape to infinity. In that case, one only need to wait for the process to enter the domain of attraction of the saddle point for the deterministic ISD.

However, as an alternative way to enforce stability in order to get quantitative results, in the present case with a simple illustrative purpose, we simply use a periodized version of the potential. Denote 
\[
\tilde U_1(x,y) = -\ln{\left(\frac{1}{2}e^{-L^2(\sin^2(x/L)+\sin^2(y/L))} + \frac{1}{2}e^{-L^2(s_x \sin^2((x-m_x)/L) + s_y \sin^2((y-m_y)/L))}\right)},
\]
where $L>0$ is a parameter, see Figure~\ref{fig:perio_pot}. This periodized potential has two supplementary saddle point, but for all $(x,y)\in \R^2$, $\lim_{L\rightarrow\infty}\tilde U_1(x,y) = U_1(x,y)$, and we may defined the ISP on $(\pi L\T)^2$.

\begin{figure}
\begin{subfigure}{.5\textwidth}
  \centering
  \includegraphics[width=.9\linewidth]{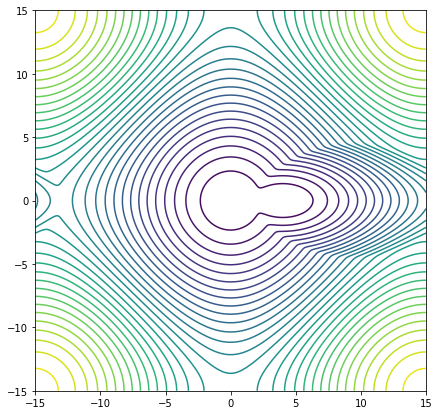}
  \caption{$(s_x,s_y)=(1,3)$.}
  \label{sfig:pot_per_(1,3)}
\end{subfigure}%
\begin{subfigure}{.5\textwidth}
  \centering
  \includegraphics[width=.9\linewidth]{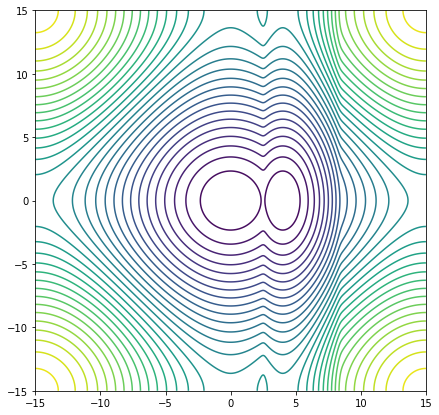}
  \caption{$(s_x,s_y)=(3,1)$.}
  \label{sfig:pot_per_(3,1)}
\end{subfigure}
\caption{Periodized potential.}
\label{fig:perio_pot}
\end{figure}

Using a 2-D histogram, $\nu=0.1$, $\varepsilon=0.05$, we get a representation of the invariant measure in Figure~\ref{fig:mesure_inv}. We see that the invariant measure charges both minima, as well as the saddle point between the two.

\begin{figure}
\begin{subfigure}{.5\textwidth}
  \centering
  \includegraphics[width=.9\linewidth]{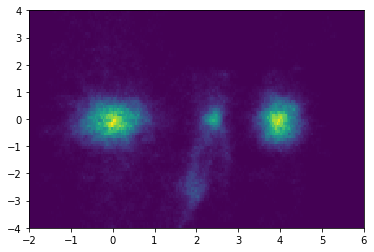}
  \caption{$(s_x,s_y)=(1,3)$.}
  \label{sfig:mesure_inv(1,3)}
\end{subfigure}%
\begin{subfigure}{.5\textwidth}
  \centering
  \includegraphics[width=.9\linewidth]{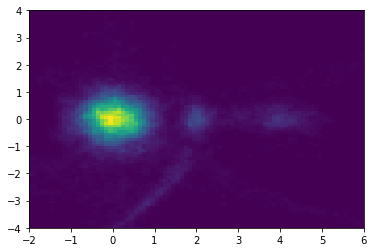}
  \caption{$(s_x,s_y)=(3,1)$.}
  \label{sfig:mesure_inv(3,1)}
\end{subfigure}
\caption{Invariant measure of the ISP for $\nu=0.1$, $\varepsilon = 0.05$.}
\label{fig:mesure_inv}
\end{figure}

Now, if the goal is to find both minima, then one can find one of them using a gradient descent, and the second one using the ISP. Since the ISD allows to find the saddle point, then if the switching rate is low enough, the switched process will do as well, and thanks to the Brownian noise, after a first switch, will have a positive probability to be in the domain of attraction for the gradient descent of the second minima, and thus will find it. If it was not in the right domain of attraction, then the process will go back to the first well, and repeat the same kind of trajectories. The success of the algorithm is then the same as the one of a rigged coin flip. For a small $\varepsilon$, as long as the transitions from one well to another are driven by this switching behaviour, the transition time behaves much better than the exponentially long time of the overdamped Langevin dynamics (with Brownian-driven rare transitions). At a fixed $\varepsilon$, when the switching rate becomes too small, the process has to wait for switching events to cross the saddle (and eventually when it becomes very small the process behaves like the overdamped Langevin dynamics and the transitions are driven by the small Brownian noise); on the other hand, a large switching rate induces an averaging phenomenon (the drift along the gentlest ascend direction vanishes), which impairs the efficiency. Thus, there should be an optimal switching rate in terms of mean transition time. This is indeed what we observed, see Figure~\ref{fig:temps_recherche_minima}.  The numerical experiment were conducted with $\varepsilon=0.03$, $(x_0,y_0)=(4,0)$, $I_0=0$, and the time was estimated by Monte-Carlo with $n=15$ repetitions.

\begin{figure}
  \centering
  \includegraphics[width=.6\linewidth]{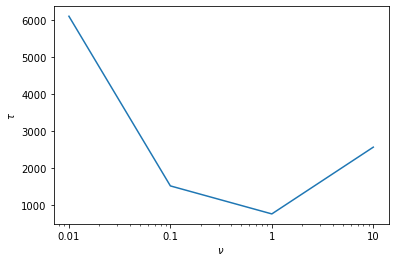}
  \caption{Time to reach a minima from the other one on the periodized potential.}
  \label{fig:temps_recherche_minima}
\end{figure}

\subsection{A 2D model with singular lines}\label{s-sec:GAD-num-2}

Now we are interested in the following potential:
\[
U_2(x,y) = (1-x^2)^2 + 2y^2,
\]
for $(x,y)\in \R^2$, see Figure~\ref{fig:level_line_singu}.

\begin{figure}
  \centering
  \includegraphics[width=.7\linewidth]{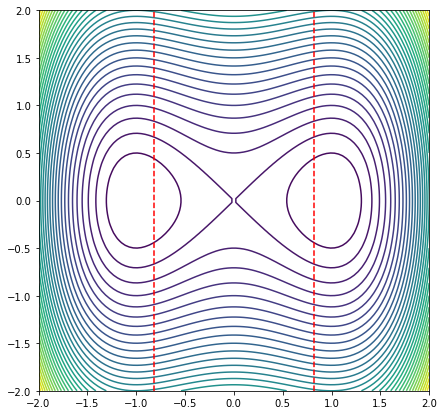}
  \caption{Level line of the potential $U_2$.}
  \label{fig:level_line_singu}
\end{figure}

The deterministic ISD and GAD have been studied for this potential in~\cite{ISD-GAD}. This potential has two minima, $(-1,0)$ and $(1,0)$, and one saddle point $(0,0)$. However, there are two lines, $\left\{x=\pm \sqrt{4/6} \right\}$, for which the Hessian of $U$ has two equal eigenvalues. The deterministic GAD and ISD  cannot cross those lines, and go to infinity, see Figure~\ref{fig:singu-U-deter}. In the noisy case, the Brownian motion may allow the process to cross the line, see Figure~\ref{fig:singu-U-cut0}.

\begin{figure}
\begin{subfigure}{.5\textwidth}
  \centering
  \includegraphics[width=.9\linewidth]{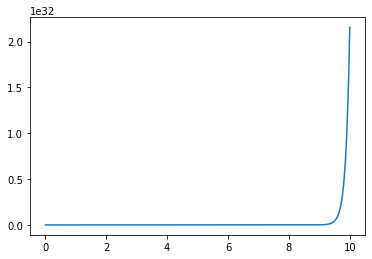}
  \caption{Evolution of $U$ along the deterministic ISD.}
  \label{sfig:ener-singu-deter}
\end{subfigure}
\begin{subfigure}{.5\textwidth}
  \centering
  \includegraphics[width=.9\linewidth]{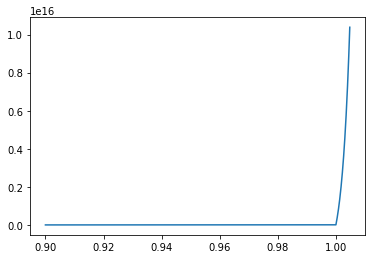}
  \caption{Trajectory of the deterministic ISD.}
  \label{sfig:traj-singu-deter}
\end{subfigure}
\caption{$\varepsilon=0.0$.}
\label{fig:singu-U-deter}
\end{figure}

\begin{figure}
\begin{subfigure}{.5\textwidth}
  \centering
  \includegraphics[width=.9\linewidth]{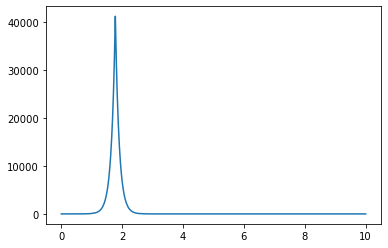}
  \caption{Evolution of $U$ along the noisy ISD.}
  \label{sfig:ener-singu-cut0}
\end{subfigure}
\begin{subfigure}{.5\textwidth}
  \centering
  \includegraphics[width=.9\linewidth]{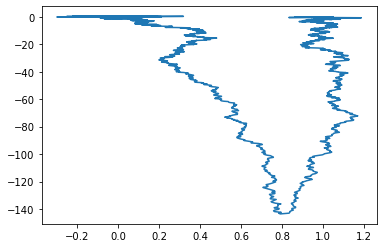}
  \caption{Trajectory of the noisy ISD.}
  \label{sfig:traj-singu-cut0}
\end{subfigure}
\caption{$\varepsilon=0.05$.}
\label{fig:singu-U-cut0}
\end{figure}

As seen in figure~\ref{fig:singu-U-cut0}, the Brownian motion allow to cross the set of singularities, but for small $\varepsilon$, this crossing may happen far from the set of minima and saddle point, because the deterministic process satisfies $\lim_{t\rightarrow\infty} y_t = +\infty$ for $x_0 >\sqrt{4/6}$ and $y_0>0$, where $(x_t,y_t)$ is the ISD. The singular lines are repulsive, and for $x>\sqrt{4/6}$, the process is attracted to the line $\left\{ x = 1 \right\}$. Concretely, for the process to cross the singular line, the noise must be great enough so that this occurs before the simulation stops due to numerical limits. 

Since the existence of this singularity results from a crossing between the eigenvalues of the Hessian of the potential $U_2$, we study the solution to solve the issue of singularities proposed in Section~\ref{s-sec:def-process}, to see in particular if it allows a faster convergences towards the saddle point in our example here. Recall the set of singularities:
\[
\mathcal S = \left\{(x,y)\in\R^2, \lambda_1(x,y)=\lambda_2(x,y) \right\},
\]
and the modified noisy ISD (writting $z=(x,y)$ and $Z_t$ the process): 
\[Z_t = -\left(I-2g(Z_t)v_1(Z_t)^tv_1(Z_t) -(2-g(Z_t)v_2(Z_t)^tv_2(Z_t))\right)\na U_2(Z_t)\dd t + \sqrt{2\varepsilon}\dd B_t,\]
with $g(z) = f(\lambda_2(z)-\lambda_1(z))$ for some function $f:\R_+\to\R_+$ such that $f(0)=1$, and $f(r)=2$ for $r\geqslant r_*$, and some small $r_*>0$. As explained, the idea is to replace the reflection with respect to $v_1(x)^{\perp}$ by a projection on $vect(v_1(x),v_2(x))^{\perp}$.


In dimension $2$, the process simply becomes a Brownian motion near $\mathcal S$. For numerical study we consider:
\[
f_1(r)=1 + \mathbbm{1}_{r>r_*},
\]
and
\[
f_2(r)=1 + \frac{r}{r_*}\mathbbm{1}_{r\leqslant r_*} + \mathbbm{1}_{r>r_*},
\]
where $r_*>0$ is a fixed parameter. An example of trajectory with $f_1$, is given in Figure~\ref{fig:singu-U-cut1}.

\begin{figure}
\begin{subfigure}{.5\textwidth}
  \centering
  \includegraphics[width=.9\linewidth]{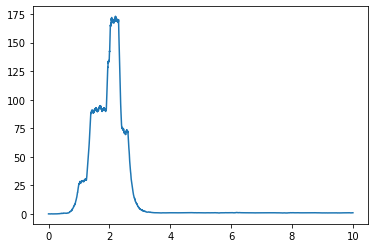}
  \caption{Time evolution of $U$ along the noisy $f_1$-modified ISD.}
  \label{sfig:ener-cut1}
\end{subfigure}
\begin{subfigure}{.5\textwidth}
  \centering
  \includegraphics[width=.9\linewidth]{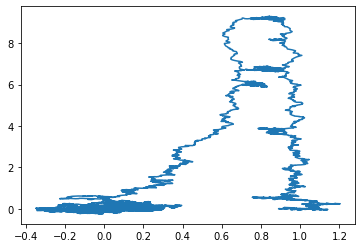}
  \caption{Trajectory of the noisy $f_1$-modified ISD.}
  \label{sfig:traj-cut1}
\end{subfigure}
\caption{$f_1$-modified ISD, $\varepsilon=0.05$.}
\label{fig:singu-U-cut1}
\end{figure}

We compare the speed at which the process reaches the saddle point with the different cut function $f_0=2$, $f_1$, and $f_2$. We chose as parameter $r_*=2$, $x_0=(0.9,0)$, and $n=500$ trials. The average of
\[
\tau = \inf\left\{t\geqslant 0, |X_t|<0.1 \right\},
\]
over the $n$ trials is displayed in Figure~\ref{fig:time-singu} for $\varepsilon\in\{0.03,0.04,0.05,0.06\}$.

\begin{figure}
  \centering
  \includegraphics[width=.7\linewidth]{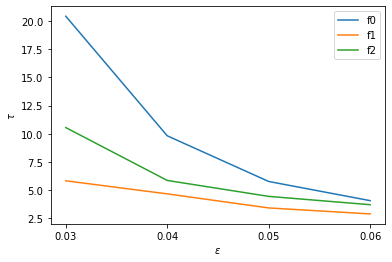}
  \caption{Time to reach the saddle point.}
  \label{fig:time-singu}
\end{figure}


We can see that the regularised dynamics  reach the saddle point faster than the initial dynamic with $f_0$. Moreover, as explained before, for smaller values of $\varepsilon$, there is a positive probability that the process doesn't reach the saddle point before hitting numerical limits. For a given temperature $\varepsilon$, this probability is lower for the processes defined from $f_1$ or $f_2$. This possible failure was not observed in the experiments displayed in Figure~\ref{fig:time-singu}, but for $\varepsilon = 0.02$, we estimated the probability of failure with 500 trials for each process. We get $0.394$ for $f_0$, and $0.046$ for $f_2$. Even at $\varepsilon=0.02$, no failure was observed for $f_1$.

Here we haven't used any of the modifications discussed at the beginning of Section~\ref{sec:Num} to enforce stability, in order to study the effects of the singular line on the time needed to find the saddle point.

\subsection{Lennard-Jones clusters}\label{s-sec:Switched-num}

We now study numerically the process for a potential in higher dimension. Consider the Lennard-Jones potential for $N=7$ particles in dimension $2$, given by 
\[
U_3(x_1,...x_N)=\sum_{i<j}W(|x_i-x_j|),
\]
where $x_i\in \R^2$ for all $1\leqslant i \leqslant N$, and for $r>0$:
\[
W(r) = 4\left( \frac{1}{r^{12}}-\frac{1}{r^6} \right).
\]
The potential $U$ is invariant by rotation and translation of the full system and by permutation of the particles. Once these symmetries are ruled out,
it has three non-global local minima, and an additional global one, represented in Figure~\ref{fig:LJ-minima}. Lennard-Jones clusters in dimension 2 or 3 are classical models in physics and have been extensively, see e.g. \cite{Northby87,DITTNER2017} and references within.

\begin{figure}
\begin{subfigure}{.5\textwidth}
  \centering
  \includegraphics[width=.9\linewidth]{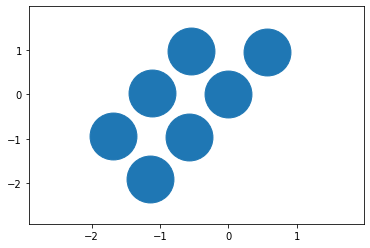}
  \caption{$U\approx -11,40$}
  \label{sfig:U=-11,40}
\end{subfigure}%
\begin{subfigure}{.5\textwidth}
  \centering
  \includegraphics[width=.9\linewidth]{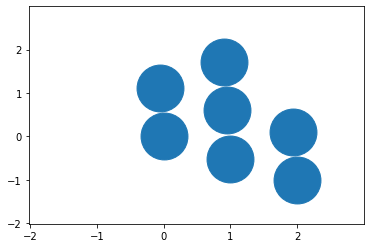}
  \caption{$U\approx -11,47$}
  \label{sfig:U=-11,47}
\end{subfigure}
\\
\begin{subfigure}{.5\textwidth}
  \centering
  \includegraphics[width=.9\linewidth]{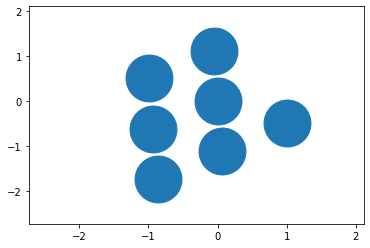}
  \caption{$U\approx -11,50$}
  \label{sfig:U=-11,50}
\end{subfigure}%
\begin{subfigure}{.5\textwidth}
  \centering
  \includegraphics[width=.9\linewidth]{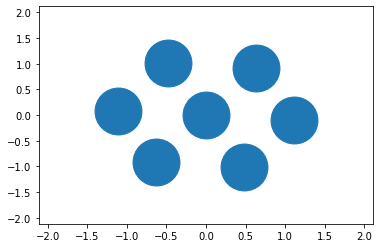}
  \caption{$U\approx -12,53$}
  \label{sfig:U=-12,53}
\end{subfigure}
\caption{List of minimizers of $U$}
\label{fig:LJ-minima}
\end{figure}

We use as initial condition the local minimum such that $U\approx -11,47$, and we compute the time necessary for the process to visit all minima for different values of $\nu$ and $\varepsilon$, estimated over 20 experiments. The average time is given in Figure~\ref{fig:time-LJ}. These results confirm the interpretation of  Figure~\ref{fig:temps_recherche_minima} discussed in Section~\ref{s-sec:GAD-num-1}. We see that the sensibility of the exploration time to the switching rate increases as the temperature $\varepsilon$ decreases, as should be expected: in the high frequency regime, the drift along the gentlest ascent direction averages to zero so that, in that direction, the process moves at the speed of a Brownian motion with variance $\varepsilon$; in the low frequency regime, the transitions between different wells are driven by the Brownian noise and thus follow an Eyring-Kramers formula, so that the mean transition time is exponentially large with $\varepsilon^{-1}$.

Since $U_3$ is  invariant by global translations and rotations of the system, $0$ is always an eigenvalue of  its Hessian, associated to eigenvectors which are orthogonal to $\na U_3$. As a consequence, at points $x$ where $\nabla^2 U_3(x)$ has no negative eigenvalue, the ISD dynamics is the same as the gradient descent dynamics. For a real application, these known symmetries should be taken into account in order to avoid this. However, as we saw, even without addressing this issue,  the ISP still finds all minima faster than the Langevin process: even though the gentle ascend behaviour plays no role in the vicinity of minimizers, it still makes saddle points locally attractive for the dynamics, which is a key point for the efficiency of the exploration.

\subsection*{Acknowledgement.} This works is supported by the French ANR grant SWIDIMS (ANR-20-CE40-0022).


\begin{figure}
  \centering
  \includegraphics[width=.7\linewidth]{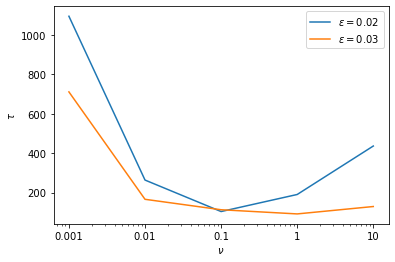}
  \caption{Time to visit all minima of the Lennard-Jones cluster.}
  \label{fig:time-LJ}
\end{figure}

\bibliography{biblioswitch}
\bibliographystyle{plain}

\end{document}